\documentclass[reqno]{amsart}

\usepackage{enumerate}
\usepackage{amsmath}%
\usepackage{amsfonts}%
\usepackage{amssymb}%
\usepackage{graphicx}
\usepackage{mathrsfs}
\usepackage{hyperref}
\newtheorem{theorem}{Theorem}
\theoremstyle{plain}

\newtheorem{claim}[theorem]{Claim}

\newtheorem{construction}[theorem]{Construction}

\newtheorem{definition}[theorem]{Definition}

\newtheorem{fact}[theorem]{Fact}
\newtheorem{lemma}[theorem]{Lemma}

\newtheorem{problem}[theorem]{Problem}
\newtheorem{proposition}[theorem]{Proposition}

\numberwithin{equation}{section}
\numberwithin{theorem}{section}
\numberwithin{case}{section}

\numberwithin{subcase}{case}

\def\F{\mathcal{F}}

\def \a{\alpha}
\def \e{\epsilon}
\def \r{\gamma}

\def \bfi{\mathbf{i}}
\def \bfu{\mathbf{u}}
\def \bfv{\mathbf{v}}
\def \calP{\mathcal{P}}

\def \PM{\textbf{PM}}
\def \DPM{\textbf{DPM}}

\begin{document}
\title{Decision problem for Perfect Matchings in Dense $k$-uniform Hypergraphs}
\author{Jie Han}
\address{Instituto de Matem\'{a}tica e Estat\'{\i}stica, Universidade de S\~{a}o Paulo, Rua do Mat\~{a}o 1010, 05508-090, S\~{a}o Paulo, Brazil}
\email{jhan@ime.usp.br}
\thanks{The author is supported by FAPESP (2014/18641-5, 2015/07869-8).}

\date{\today}

\begin{abstract}
For any $\r>0$, Keevash, Knox and Mycroft \cite{KKM13} constructed a polynomial-time algorithm which determines the existence of perfect matchings in any $n$-vertex $k$-uniform hypergraph whose minimum codegree is at least $n/k+\r n$. 
We prove a structural theorem that enables us to determine the existence of a perfect matching for any $k$-uniform hypergraph with minimum codegree at least $n/k$. This solves a problem of Karpi\'nski, Ruci\'nski and Szyma\'nska completely.
Our proof uses a lattice-based absorbing method.
\end{abstract}

\maketitle

\section{Introduction}

Given $k\ge 2$, a $k$-uniform hypergraph (in short, \emph{$k$-graph}) $H=(V(H), E(H))$ consists of a vertex set $V(H)$ and an edge set $E(H)\subseteq \binom{V(H)}{k}$, where every edge is a $k$-element subset of $V(H)$. A \emph{matching} in $H$ is a collection of vertex-disjoint edges of $H$. A \emph{perfect matching} $M$ in $H$ is a matching that covers all vertices of $H$. Throughout this note, we assume that $k$ divides $|V(H)|$, which is clearly a necessary condition for the existence of a perfect matching in $H$. 

The question of whether a given $k$-graph $H$ contains a perfect matching is one of the most fundamental questions of combinatorics. In the graph case $k=2$, Tutte's Theorem \cite{Tu47} gives necessary and sufficient conditions for $H$ to contain a perfect matching, and Edmonds' Algorithm~\cite{Edmonds} finds such a matching in polynomial time. However, for $k\ge 3$ this problem was one of Karp's celebrated 21 NP-complete problems \cite{Karp}. Since the general problem is intractable provided P $\neq$ NP, it is natural to ask conditions on $H$ which make the problem tractable or even guarantee that a perfect matching exists. One well-studied class of such conditions are minimum degree conditions.

\subsection{Perfect matchings under minimum degree conditions}
Given a $k$-graph $H$ with a set $S$ of $d$ vertices (where $1 \le d \le k-1$) we define $\deg_{H} (S)$ to be the number of edges containing $S$ (the subscript $H$ is omitted if it is clear from the context). The \emph{minimum $d$-degree $\delta _{d} (H)$} of $H$ is the minimum of $\deg_{H} (S)$ over all $d$-vertex sets $S$ in $H$.  
We refer to $\delta_{k-1} (H)$ as the \emph{minimum codegree} of $H$.

Over the last few years there has been a strong focus in establishing minimum $d$-degree thresholds that force a perfect matching in a $k$-graph \cite{AFHRRS, CzKa, HPS, Khan2, Khan1, KO06mat, KOT, MaRu, Pik, RR, RRS06mat, RRS09, TrZh12, TrZh13}. In particular, R\"odl, Ruci\'nski and Szemer\'edi \cite{RRS09} determined the minimum codegree threshold that ensures a perfect matching in a $k$-graph on $n$ vertices for large $n$ and all $k\ge 3$. The threshold is $n/2-k+C$, where $C\in\{3/2, 2, 5/2, 3\}$ depends on the values of $n$ and $k$. 
In contrast, they proved that a $k$-graph $H$ on $n$ vertices satisfying $\delta_{k-1}(H)\ge n/k+O(\log n)$ contains a matching of size $n/k-1$ (one edge away from a perfect matching). Recently the author \cite{Han14_mat} improved this result by showing that $\delta_{k-1}(H)\ge n/k-1$ suffices. The following construction, usually called \emph{space barrier}, shows that this is best possible.

\begin{construction}[Space Barrier]\label{con:sb}
Let $V$ be a set of size $n$ and fix $S\subseteq V$ with $|S|<n/k$. Let $H$ be the $k$-graph whose edges are all $k$-sets that intersect $S$.
\end{construction}

Note that the minimum codegree of $H$ is $|S|$ and any matching in Construction \ref{con:sb} has at most $|S|$ edges.

\medskip
Let $\PM(k,\delta)$ be the decision problem of determining whether a $k$-graph $H$ with $\delta_{k-1}(H)\ge \delta n$ contains a perfect matching. Given the result of \cite{RRS09}, a natural question to ask is the following: For which values of $\delta$ can $\PM(k,\delta)$ be decided in polynomial time? This holds for $\PM(k,1/2)$ by the main result of \cite{RRS09}. On the other hand, $\PM(k,0)$ includes no degree restriction on $H$ at all, so is NP-complete by the result of Karp \cite{Karp}. Szyma\'nska \cite{Szy13} proved that for $\delta<1/k$ the problem $\PM(k,\delta)$ admits a polynomial-time reduction to $\PM(k,0)$ and hence $\PM(k,\delta)$ is also NP-complete. Karpi\'nski, Ruci\'nski and Szyma\'nska  showed that there exists $\e>0$ such that $\PM(k,1/2-\e)$ is in P and asked the complexity of $\PM(k,\delta)$ for $\delta\in [1/k, 1/2)$. 

\begin{problem}\cite{KRS10}\label{prob}
What is the computational complexity of $\PM(k,\delta)$ for $\delta\in [1/k, 1/2)$?
\end{problem}

Recently, Keevash, Knox and Mycroft \cite{KKM13} gave a long and involved proof that shows $\PM(k,\delta)$ is in P for any $\delta>1/k$ that leaves only $\PM(k,1/k)$ unknown. Moreover, they also constructed a polynomial-time algorithm to find a perfect matching provided one exists.
They \cite{KKM_abs} also expected that it would be difficult to solve the decision problem for $\delta=1/k$, as $n/k$ is the minimum codegree threshold at which a perfect fractional matching is guaranteed, so there is a clear behavioral change at this point.
In this paper, we give a short proof that shows $\PM(k, \delta)$ is in P for all $\delta\ge 1/k$ and thus solve Problem \ref{prob} completely.

\begin{theorem}\label{thm:main}
Fix $k\ge 3$. Let $H$ be an $n$-vertex $k$-graph with $\delta_{k-1}(H)\ge n/k$. Then there is an algorithm with running time $O(n^{3k^2 - 5k})$, which determines whether $H$ contains a perfect matching.
\end{theorem}

The proof of Theorem \ref{thm:main} follows the approach of \cite{KKM13}, from which we use several definitions and results. The heart of the algorithm in that paper was a structural theorem \cite[Theorem 1.10]{KKM13}, which was proved by partitioning the $k$-graph $H$ into a number of $k$-partite $k$-graphs, before finding a perfect matching in each of these $k$-partite $k$-graphs by using a theorem of Keevash and Mycroft \cite{KM1}. Our main improvement is to replace this by a new structural theorem (Theorem \ref{thm:PM}) which significantly simplifies the argument in \cite{KKM13}, and which applies in the exact case $\delta_{k-1}(H)\ge n/k$ (the structural theorem of \cite{KKM13} only applied for $\delta_{k-1}(H)\ge n/k + o(n)$). This already provides a polynomial-time algorithm deciding the existence of perfect matchings, and a faster algorithm as claimed in Theorem \ref{thm:main} is obtained by combining Theorem \ref{thm:PM} with ideas from \cite{KKM13}. Our proof of Theorem \ref{thm:PM} uses a lattice-based absorbing method which does not need the hypergraph regularity lemma or the main result of \cite{KM1}. This novel approach, which combines the powerful absorbing technique with the `divisibility barrier' structures considered in \cite{KM1}, may well be useful for other matching and tiling problems in hypergraphs.

\subsection{Lattice-based constructions}

It is shown in \cite{KM1} that a $k$-graph $H$ has a perfect matching or is close to a family of lattice-based constructions termed ``divisibility barriers''. The following examples of divisibility barriers were given in \cite{RRS09}.

\begin{construction}\label{con12}
Let $X$ and $Y$ be disjoint sets such that $|X\cup Y|=n$ and $|X|$ is odd, and let $H$ be the $k$-graph on $X\cup Y$ whose edges are all $k$-sets which intersect $X$ in an even number of vertices.
\end{construction}

\begin{construction}\label{con13}
Let $X$ and $Y$ be disjoint sets such that $|X\cup Y|=n$ and $|X|-n/k$ is odd, and let $H$ be the $k$-graph on $X\cup Y$ whose edges are all $k$-sets which intersect $X$ in an odd number of vertices.
\end{construction}

To see why there is no perfect matching in Construction \ref{con13}, note that a perfect matching has $n/k$ edges, intersecting $X$ in $n/k$ (mod 2) number of vertices. Since $|X|\not\equiv n/k$ (mod 2), a perfect matching does not exist.
To describe divisibility barriers in general, we make the following definition.
In this paper, every partition has an implicit order on its parts. 

\begin{definition}
Let $H=(V, E)$ be a $k$-graph and let $\calP$ be a partition of $V$ into $d$ parts. Then the \emph{index vector} $\bfi_{\calP}(S)\in \mathbb{Z}^d$ of a subset $S\subseteq V$ with respect to $\calP$ is the vector whose coordinates are the sizes of the intersections of $S$ with each part of $\calP$, namely, $\bfi_{\calP}(S)_X=|S\cap X|$ for $X\in \calP$. Furthermore,
\begin{enumerate}[\emph{(}i\emph{)}]
\item $I_\calP(H)$ denotes the set of index vectors $\bfi_\calP (e)$ of edges $e\in H$, and
\item $L_{\calP}(H)$ denotes the lattice (i.e. additive subgroup) in $\mathbb{Z}^d$ generated by $I_\calP(H)$.
\end{enumerate}
\end{definition}

A \emph{divisibility barrier} is a $k$-graph $H$ which admits a partition $\calP$ of its vertex set $V$ such that $\bfi_{\calP}(V)\notin L_{\calP}(H)$; To see that such an $H$ contains no perfect matching, let $M$ be a matching in $H$. Then $\bfi_\calP(V(M))=\sum_{e\in M}\bfi_\calP(e)\in L_\calP(H)$. But $\bfi_\calP(V)\notin L_\calP(H)$, so $V(M)\neq V$, namely, $M$ is not perfect.
For example, to see that this generates Construction \ref{con12}, let $\calP$ be the partition into parts $X$ and $Y$; then $L_{\calP}(H)$ is the lattice of vectors $(x,y)$ in $\mathbb{Z}^2$ for which $x$ is even and $k$ divides $x+y$, and $|X|$ being odd implies that $\bfi_\calP (V)\notin L_\calP (H)$.

\section{The Main structural theorem}

We need the following definitions from \cite{KKM13} before giving the statement of our structural theorem. 

\begin{definition}\cite{KKM13}
Suppose $L$ is a lattice in $\mathbb{Z}^d$.
\begin{enumerate}[\emph{(}i\emph{)}]
\item We say that $\bfi\in \mathbb{Z}^d$ is an \emph{$r$-vector} if it has non-negative coordinates that sum to $r$. We write $\bfu_j$ for the `unit' 1-vector that is 1 in coordinate $j$ and 0 in all other coordinates.
\item We say that $L$ is an \emph{edge-lattice} if it is generated by a set of $k$-vectors.
\item We write $L_{\max}^d$ for the lattice generated by all $k$-vectors. So $L_{\max}^d=\{ x\in \mathbb{Z}^d: k \text{ divides }\sum_{i\in [d]} x_i \}$.
\item We say that $L$ is \emph{complete} if $L=L_{\max}^d$, otherwise it is \emph{incomplete}.
\item A \emph{transferral} is a non-zero difference $\bfu_i - \bfu_j$ of 1-vectors.
\item We say that $L$ is \emph{transferral-free} if it does not contain any transferral.
\item We say that a set $I$ of $k$-vectors is \emph{full} if for every $(k-1)$-vector $\mathbf{v}$ there is some $i\in [d]$ such that $\mathbf{v}+\bfu_i\in I$.
\item We say that $L$ is \emph{full} if it contains a full set of $k$-vectors and is transferral-free.
\end{enumerate}
\end{definition}

We recall the following construction \cite[Construction 1.6]{KKM13} in the case when $k=4$.

\begin{construction}\cite{KKM13}\label{con:333}
Let $\calP=\{V_1, V_2, V_3\}$ be a partition of vertex set $|V|=n$, with $|V_1|=n/3-2$, $|V_2|=n/3$ and $|V_3|=n/3+2$. Fix some vertex $x\in V_2$, and let $H$ be the $4$-graph such that $E(H)$ consists of all $k$-sets $e$ with $\bfi_{\calP}(e)=(3,0,1), (0,3,1), (0,0,4), (2,2,0)$ or $(1,1,2)$ and all $k$-sets $e$ containing $x$ with $\bfi_{\calP}(e)=(0,1,3)$.
\end{construction}

Note that $\delta_3(H) = n/3 - 4$. It is not hard to see that $\bfi_{\calP}(V)\in L_{\calP}(H)$ but $H$ does not contain a perfect matching. Indeed, if a matching $M$ in $H$ does not contain any edge $e$ with index vector $(0,1,3)$, then $|V(M)\cap V_2| - |V(M)\cap V_1|\equiv 0$ (mod 3). Otherwise $M$ contains an edge with index vector $(0,1,3)$, thus we have $|V(M)\cap V_2| - |V(M)\cap V_1|\equiv 1$ (mod 3). In either case, $M$ is not perfect since $|V_2| - |V_1|=2$. 
In fact, as shown in \cite{KKM13}, $\bfi_{\calP}(V)\in L_{\calP}(H)$ holds for any $\calP$ of $V(H)$. Thus, having a divisibility barrier is not a necessary condition for $H$ not containing a perfect matching.

Note that when we determine if $\bfi_{\calP}(V)\in L_{\calP}(H)$, we are free to use any multiple of any vectors $\bfi\in I_{\calP}(H)$. But in Construction \ref{con:333}, all edges $e$ with $\bfi_{\calP}(e)=(0,1,3)$ contain $x$, thus a matching in $H$ can only contain one edge with index vector $(0,1,3)$. So although $\bfi_{\calP}(V)\in L_{\calP}(H)$, there is no perfect matching. 
Thus, it is natural to consider the following robust edge-lattice such that for every $k$-vector $\bfi\in I_{\calP}^{\mu}(H)$, there are many edges $e$ such that $\bfi_{\calP}(e)= \bfi$.

\begin{definition}[Robust edge-lattices] Let $H=(V, E)$ be a $k$-graph and $\calP$ be a partition of $V$ into $d$ parts. Then for any $\mu>0$,
\begin{enumerate}[(i)]
\item $I_{\calP}^\mu(H)$ denotes the set of all $\bfi\in \mathbb{Z}^d$ such that at least $\mu |V|^k$ edges $e\in H$ have $\bfi_{\calP}(e)=\bfi$.
\item $L_{\calP}^{\mu}(H)$ denotes the lattice in $\mathbb{Z}^d$ generated by $I_{\calP}^\mu(H)$.
\end{enumerate}
\end{definition}

We will show that there exists a partition $\calP$ of $V(H)$ and $\mu>0$, such that if $\bfi_{\calP}(V)\in L_{\calP}^{\mu}(H)$, then $H$ contains a perfect matching. 
Indeed, even a weaker condition suffices. If we can find a small matching $M$ such that $\bfi_{\calP}(V\setminus V(M))\in L_{\calP}^{\mu}(H[V\setminus V(M)])=L_{\calP}^{\mu}(H)$, then we can apply our proof above to show that $H[V\setminus V(M)]$ contains a perfect matching $M'$. Thus $M\cup M'$ is a perfect matching of $H$. Note that we can guarantee $L_{\calP}^{\mu}(H[V\setminus V(M)])=L_{\calP}^{\mu}(H)$ by selecting $\mu$ `wisely' and requiring that $M$ is small.
The following definitions are essentially from \cite{KKM13}. The only difference is that a full pair defined in \cite{KKM13} has at most $k-1$ parts. 

\begin{definition}\cite{KKM13}
Let $H=(V, E)$ be a $k$-graph.
\begin{enumerate}[\emph{(}i\emph{)}]
\item A \emph{full pair} $(\mathcal{P}, L)$ for $H$ consists of a partition $\calP$ of $V$ into $d\le k$ parts and a full edge-lattice $L\subset \mathbb{Z}_d$.
\item A (possibly empty) matching $M$ of size at most $|\calP|-1$ is a \emph{solution} for $(\calP, L)$ (in $H$) if $\bfi_{\calP}(V\setminus V(M))\in L$; we say that $(\calP, L)$ is \emph{soluble} if it has a solution, otherwise \emph{insoluble}.
\end{enumerate}
\end{definition}

The following lemma provides a partition $\calP_0$ such that we can develop the absorbing lemma on the pair $(\calP_0, L_{\calP_0}^\mu(H))$ for some $\mu>0$. 
For a small enough $\mu>0$, $I_{\calP_0}^\mu(H)$ is full. 
However, the pair $(\calP_0, L_{\calP_0}^\mu(H))$ may not be full because it may contain transferrals. Then we will obtain a full pair $(\calP_0', L_{\calP_0'}^\mu(H))$ from the pair $(\calP_0, L_{\calP_0}^\mu(H))$ by iteratively merging parts that contain transferrals. 

We use the reachability arguments introduced by Lo and Markstr\"om \cite{LM2, LM1}.
We say that two vertices $u$ and $v$ are \emph{$(\beta, i)$-reachable} in $H$ if there are at least $\beta n^{i k-1}$ $(i k-1)$-sets $S$ such that both $H[S\cup \{u\}]$ and $H[S\cup \{v\}]$ have perfect matchings. 
In this case, we call $S$ a \emph{reachable set} for $u$ and $v$.
We say that a vertex set $U$ is \emph{$(\beta, i)$-closed in $H$} if any two vertices $u,v\in U$ are $(\beta, i)$-reachable in $H$.
For two partitions $\calP, \calP'$ of a set $V$, we say that $\calP$ \emph{refines} $\calP'$ if every vertex class of $\calP$ is a subset of some vertex class of $\calP'$. 
Throughout this paper, $x\ll y$ means that for any $y> 0$ there exists $x_0> 0$ such that for any $x< x_0$ the following statement holds.

\begin{lemma}\label{lem:PL}
Given an integer $k\ge 3$, for any $0<\r \ll 1/k$, suppose that
$1/n \ll \{\beta, \mu\} \ll \r$.
Then for each $k$-graph $H$ on $n$ vertices with $\delta_{k-1}(H)\ge n/k-\r n$, we find partitions $\calP_0=\{V_1, \dots, V_d \}$ and $\calP_0'=\{V_1', \dots, V_{d'}'\}$ of $V(H)$ in time $O(n^{2^{k-1}k+1})$ satisfying the following properties:
\begin{enumerate}[\emph{(}i\emph{)}]
\item $\calP_0$ refines $\calP_0'$ and $(\calP_0', L_{\calP_0'}^{\mu}(H))$ is a full pair,
\item each partition set of $\calP_0$ or $\calP_0'$ has size at least $n/k-2\r n$,
\item for each $D\subseteq V(H)$ such that $\bfi_{\calP_0'}(D)\in L_{\calP_0'}^\mu(H)$, we have $\bfi_{\calP_0}(D)\in L_{\calP_0}^\mu(H)$,
\item for each $i\in [d]$, $V_i$ is $(\beta, 2^{k-1})$-closed in $H$.
\end{enumerate}
\end{lemma}

Given integers $n\ge k\ge 3$, let $\mathcal{H}_{n,k}$ be the collection of $k$-graphs $H$ such that there is a partition of $V(H)= X\cup Y$ such that $n/k-|X|$ is odd and all edges of $H$ intersect $X$ at an odd number of vertices. Note that the members of $\mathcal{H}_{n,k}$ are subhypergraphs of the $k$-graphs in Construction \ref{con13} and thus none of them has a perfect matching.

Now we are ready to state our main structural theorem.

\begin{theorem}\label{thm:PM}
Fix an integer $k\ge 3$. Suppose
\[
1/n_0 \ll \{\beta, \mu\} \ll \r \ll 1/k.
\]
Let $H$ be a $k$-graph on $n\ge n_0$ vertices such that $\delta_{k-1}(H)\ge n/k$ with $\calP_0$ and $\calP_0'$ found by Lemma~\ref{lem:PL}. Then $H$ contains a perfect matching if and only if the full pair $(\calP_0', L_{\calP_0'}^{\mu}(H))$ is soluble and $H\notin \mathcal H_{n,k}$.
\end{theorem}

We first prove the forward implication. 
The following lemma from \cite{KKM13} says that we can omit the condition on the size of $M$ when considering solubility. Although the definition of full pairs is slightly different in \cite{KKM13}, the same proof works in our case.

\begin{lemma}\cite[Lemma 6.9]{KKM13}\label{lem69}
Let $(\calP, L)$ be a full pair for a $k$-graph $H$, where $k\ge 3$. Then $(\calP, L)$ is soluble if and only if there exists a matching $M$ in $H$ such that $\bfi_{\calP}(V(H)\setminus V(M))\in L$.
\end{lemma}

\begin{proof}[Proof of the forward implication of Theorem~\ref{thm:PM}]
If $H$ contains a perfect matching $M$, then $\bfi_{\calP_0'}(V(H)\setminus V(M))=\mathbf{0}\in L_{\calP_0'}^{\mu}(H)$. Since $(\calP_0', L_{\calP_0'}^{\mu}(H))$ is a full pair, by Lemma \ref{lem69}, it is soluble. Furthermore, $H\notin \mathcal H_{n,k}$ because no member of $\mathcal H_{n,k}$ contains a perfect matching.
\end{proof}

The proof of the backward implication is more involved. 
For this purpose, we develop a lattice-based absorbing method.
In order to use the absorbing method, we need to reserve $O(\log n)$ vertices for our absorbing matching and then look for an almost perfect matching in the remaining $k$-graph $H'$. But an almost perfect matching may not exist if $H'$ is close to the space barrier (Construction \ref{con:sb}). This means that our absorbing technique works only if $H$ is not extremal (not close to the space barrier). So we separate the proof into an \emph{extremal case} and a \emph{non-extremal case} and then handle the extremal case separately.
More precisely, we say that $H$ is \emph{$\r$-extremal} if $V(H)$ contains an independent subset of order at least $(1-\r)\frac{k-1}k n$. By picking  constants $0<\r, \e\ll 1/k$ such that $\e=11k\r$, the backward implication follows from the following two theorems immediately.

\begin{theorem}\label{thm2}
For any $0<\r \ll 1/k$, suppose that
$1/n \ll \{\beta, \mu\} \ll \r$.
Let $H$ be a $k$-graph on $n$ vertices such that $\delta_{k-1}(H)\ge n/k - \r n$ with $\calP_0$ and $\calP_0'$ found by Lemma \ref{lem:PL}. Moreover, if $H$ is not $11k\r$-extremal and $(\calP_0', L_{\calP_0'}^{\mu}(H))$ is soluble, then $H$ contains a perfect matching.
\end{theorem}

\begin{theorem}\label{thm3}
For any $0<\e \ll 1/k$ and sufficiently large integer $n$ the following holds. Suppose $H$ is a $k$-graph on $n$ vertices such that $\delta_{k-1}(H)\ge n/k$ and $H$ is $\e$-extremal. If $H\notin \mathcal H_{n,k}$, then $H$ contains a perfect matching.
\end{theorem}

Note that we only need that $(\calP_0', L_{\calP_0'}^{\mu}(H))$ is soluble in the non-extremal case  and $H\notin \mathcal H_{n,k}$ in the extremal case.

Let us compare our method and the traditional absorbing method and outline our proof of Theorem \ref{thm2}.
The absorbing method, initiated by R\"odl, Ruci\'nski and Szemer\'edi \cite{RRS06}, has been shown to be efficient in finding spanning structures in graphs and hypergraphs. 
For example, in order to get a perfect matching in a $k$-graph $H$, it is first shown that any $k$-set has many absorbing sets in $H$. Then we apply the probabilistic method to find a small matching that can absorb any (much smaller) collection of $k$-vertex sets. 

However, with the potential divisibility barriers, we cannot guarantee that every $k$-vertex set can be absorbed in general unless the minimum codegree is at least $(1/2+\r)n$. 
In this paper, we develop a lattice-based absorbing method to overcome this difficulty. 
More precisely, we first find a partition $\calP_0=\{V_1,\dots, V_d\}$ of $V(H)$ such that any two vertices from the same $V_i$ are reachable in $H$ (property (iv) of Lemma \ref{lem:PL}). 
Then we build our absorbing matching that can absorb any $k$-set $S$ with index vector $\bfi_{\calP_0}(S)\in I_{\calP_0}^\mu(H)$. 
After applying the almost perfect matching theorem (Theorem \ref{thm:next}), we will have only $k$ vertices left unmatched. 
Then the solubility condition guarantees that we can release some edges from the partial matching such that the set of unmatched vertices can be partitioned into $k$-sets $S_1,\dots, S_{d''}$ for some constant $d''$ such that $\bfi_{\calP_0}(S_i)\in I_{\calP_0}^\mu(H)$, so we can absorb them by the absorbing matching and get a perfect matching of $H$.

The rest of the paper is organized as follows. We prove Theorem \ref{thm2} in Section 3 and prove Theorem \ref{thm3} in Section 4, respectively. We show the algorithms and prove Theorem~\ref{thm:main} in Section 5. We give concluding remarks at the end.

\section{The Non-extremal Case}

In this section we prove Theorem \ref{thm2}.

\subsection{Tools}

We use the following theorem from \cite{Han14_mat} which is stated slightly differently, but can be easily derived from the original statement.

\begin{theorem}\cite[Theorem 1.4]{Han14_mat}\label{thm:next}
Suppose that $1/n\ll \r \ll 1/k$ and $n\in k\mathbb{N}$. Let $H$ be a $k$-graph on $n$ vertices with $\delta_{k-1}(H)\ge n/k - \r n$. If $H$ is not $5k\r$-extremal, then $H$ contains a matching that leaves $k$ vertices uncovered.
\end{theorem}

Although we are one step away from a perfect matching after applying Theorem \ref{thm:next}, it is not easy to finish the last edge (in many cases impossible).
Let us introduce the following definition and result in \cite{KKM13}.

\begin{definition}
Suppose $L$ is an edge-lattice in $\mathbb{Z}^{|\calP|}$, where $\calP$ is a partition of a set $V$.
\begin{enumerate}[\emph{(}i\emph{)}]
\item The \emph{coset group} of $(\calP, L)$ is $G=G(\calP, L)=L_{\max}^{|\calP|}/L$.
\item For any $\bfi\in L_{\max}^{|\calP|}$, the \emph{residue} of $\bfi$ in $G$ is $R_G(\bfi)=\bfi+L$. For any $A\subseteq V$ of size divisible by $k$, the \emph{residue} of $A$ in $G$ is $R_G(A)=R_G(\bfi_{\calP}(A))$.
\end{enumerate}
\end{definition}

\begin{lemma}\cite[Lemma 6.4]{KKM13}\label{lem64}
If $k\ge 3$ and $L$ is a full lattice, then $|G(\calP, L)|=|\calP|$.
\end{lemma}

Suppose $I$ is a set of $k$-vectors of $\mathbb{Z}^d$.
Let $J(I)$ be the set of $l$-vectors $\bfi$ for $k\le l\le k^2$ such that $\bfi$ can be written as a linear combination of vectors in $I$, namely, there exists $a_{\bfv}(\bfi)\in \mathbb{Z}$ for all $\bfv\in I$ such that
\begin{equation}\label{eq:Cmax}
\bfi=\sum_{\bfv\in I}a_{\bfv}(\bfi)\bfv. 
\end{equation}
We denote $C(d, k, I)$ as the maximum of $|a_{\bfv}(\bfi)|, \bfv\in I$ over all $\bfi\in J(I)$ and denote $C(k):=\max C(d,k,I)$ over all sets $I$ of $k$-vectors of $\mathbb{Z}^d$ and all $1\le d\le k$.

Fix an integer $i>0$. For a $k$-vertex set $S$, we say a set $T$ is an \emph{absorbing $i$-set for $S$} if $|T|=i$ and both $H[T]$ and $H[T\cup S]$ contain perfect matchings. 
Now we may state our absorbing lemma.

\begin{lemma}[Absorbing Lemma]\label{lem:abs}
Suppose $1\le d\le k$ and
\[
1/n \ll 1/c \ll \{\beta, \mu\} \ll 1/k, 1/t, C(k)
\]
Suppose that $\calP_0=\{V_1, \dots, V_d\}$ is a partition of $V(H)$ such that for $i\in [d]$, $V_i$ is $(\beta, t)$-closed in $H$.
Then there is a family $\F_{abs}$ of disjoint $tk^2$-sets with size at most $c\log n$ such that $H[V(F)]$ contains a perfect matching for all $F\in \F_{abs}$ and every $k$-vertex set $S$ with $\bfi_{\calP_0}(S)\in I_{\calP_0}^{\mu} (H)$ has at least $2k^k C(k)$ absorbing $t k^2$-sets in $\F_{abs}$. 
\end{lemma}

We postpone the proof of the absorbing lemma to the end of this section and prove the main goal of this section, Theorem \ref{thm2} first.

\subsection{Proof of Theorem \ref{thm2}}

\begin{proof}[Proof of Theorem \ref{thm2}]
Fix $0<\r\ll 1/k$.
Write $C=C(k)$ and suppose
\[
1/n \ll 1/c\ll \{\beta, \mu\} \ll \r, C.
\]
Let $H$ be a $k$-graph on $n$ vertices such that $\delta_{k-1}(H)\ge n/k - \r n$ with $\calP_0$ and $\calP_0'$ found by Lemma \ref{lem:PL} satisfying properties (i)-(iv). Moreover, assume that $H$ is not $11k\r$-extremal and $(\calP_0', L_{\calP_0'}^{\mu}(H))$ is soluble.
Let $\calP_0=\{V_1,\dots, V_d\}$ and $\calP_0'=\{V_1', \dots, V_{d'}'\}$ and note that $d'\le d\le k$ by (ii).
We first apply Lemma \ref{lem:abs} on $H$ with $t=2^{k-1}$ and get a family $\F_{abs}$ of $2^{k-1}k^2$-sets with size at most $c\log n$ such that every $k$-set $S$ of vertices with $\bfi_{\calP_0}(S)\in I_{\calP_0}^{\mu}(H)$ has at least $2k^k C$ absorbing $2^{k-1} k^2$-sets in $\F_{abs}$. 

Since $(\calP_0', L_{\calP_0'}^{\mu}(H))$ is soluble, there exists a matching $M_1$ of size at most $d'-1$ such that $\bfi_{\calP_0'}(V(H)\setminus V(M_1))\in L_{\calP_0'}^{\mu}(H)$. Note that $V(M_1)$ may intersect $V(\F_{abs})$, but $M_1$ can only intersect at most $k(k-1)$ absorbing sets of $\F_{abs}$. Let $\F_{0}$ be the subfamily of $\F_{abs}$ obtained from removing the $2^{k-1}k^2$-sets that intersect $V(M_1)$. Let $M_0$ be the perfect matching on $V(\F_0)$ that consists of perfect matchings on each member of $\F_0$. Note that every $k$-set $S$ of vertices with $\bfi_{\calP_0}(S)\in I_{\calP_0}^{\mu}(H)$ has at least $2k^k C - k(k-1)$ absorbing sets in $\F_{0}$.

Now we switch to $\calP_0$. We want to `store' some edges for each $k$-vector in $I_{\calP_0}^{\mu}(H)$ for future use. More precisely, we find a matching $M_2$ in $V(H)\setminus V(M_{0}\cup M_1)$ which contains $C$ edges $e$ with $\bfi_{\calP_0}(e)=\bfi$ for every $\bfi\in I_{\calP_0}^{\mu}(H)$. So $|M_2|\le \binom{k+d-1}{k}C$ and by the greedy algorithm, the process is possible because $H$ contains at least $\mu n^k$ edges for each $k$-vector $\bfi\in I_{\calP_0}^{\mu}(H)$ and $|V(M_0\cup M_{1}\cup M_2)|\le 2^{k-1} k^2 c\log n + k(k-1) + \binom{k+d-1}{k}C < \mu n$.

\medskip
Let $H':=H[V(H)\setminus V(M_0\cup M_{1} \cup M_2)]$. Note that $|V(H')|\ge n - \mu n$. So
\[
\delta_{k-1}(H') \ge \delta_{k-1}(H) - \mu n \ge n/k - 2\r n \ge (1/k - 2\r)|V(H')|.
\]
Moreover, if $H'$ is $10k\r$-extremal, namely, $V(H')$ contains an independent subset of order at least
\[
(1-10k\r)\frac{k-1}{k}|V(H')| \ge (1-10k\r)\frac{k-1}{k} (n-\mu n) \ge (1-11k\r)\frac{k-1}{k}n,
\]
then $H$ is $11k\r$-extremal, a contradiction.
Now we can apply Theorem \ref{thm:next} on $H'$ with parameter $2\r$ in place of $\r$ and find a matching $M_3$ covering all but a set $S_0$ of $k$ vertices of $V(H')$. Note that we can absorb $S_0$ by $\F_{0}$ and get a perfect matching of $H$ immediately if $\bfi_{\calP_0}(S_0)\in I_{\calP_0}^{\mu}(H)$ (however, this may not be the case).

\medskip
Now we step back to the full pair $(\calP_0', L_{\calP_0'}^{\mu}(H))$.
Instead of index vectors, we consider the residues of $S_0$ and all edges in the matching $M_{0}\cup M_3$ with respect to $\calP_0'$. Recall that $\bfi_{\calP_0'}(V(H)\setminus V(M_1))\in L_{\calP_0'}^{\mu}(H)$. Note that, since $\calP_0$ refines $\calP_0'$, the index vectors of all edges in $M_2$ are in $I_{\calP_0'}^{\mu}(H)$. So we have $\bfi_{\calP_0'}(V(H)\setminus V(M_1\cup M_2))\in L_{\calP_0'}^{\mu}(H)$, namely, $R_G(V(H)\setminus V(M_1\cup M_2))=\mathbf{0}+ L_{\calP_0'}^{\mu}(H)$. 
Thus,
\[
 \sum_{e \in M_{0}\cup M_3} R_G(e)  +R_G(S_0) = \mathbf{0}+ L_{\calP_0'}^{\mu}(H).
\]
Suppose $R_G(S_0)=\bfv_{0}+L_{\calP_0'}^{\mu}(H)$ for some $\bfv_0\in L_{\max}^{d'}$ and we get
\[
\sum_{e \in M_{0}\cup M_3} R_G(e)  = -\bfv_0+ L_{\calP_0'}^{\mu}(H).
\]

\begin{claim}\label{clm:35}
There exist edges $e_1, \dots, e_{d''} \in M_{0}\cup M_3$ for some $d''\le d'-1$ such that 
\[
\sum_{i\in [d'']} R_G(e_i)  = -\bfv_0+ L_{\calP_0'}^{\mu}(H).
\]
\end{claim}

\begin{proof}
We follow the proof of \cite[Lemma 6.10]{KKM13}.
Fix any set of edges $e_1,\dots, e_{l} \in M_{0}\cup M_3$ for $l\ge d'$, consider $l+1$ partial sums $\sum_{i\in [j]} R_G(e_i)$ for $j=0,1,\dots, l$, where the sum equals $\mathbf{0} + L_{\calP_0'}^{\mu}(H)$ when $j=0$. Since $G=G(\calP, L)$ is a group, the sums are still in $G$. By Lemma \ref{lem64}, $|G|=|\calP_0'|=d'$, then by the pigeonhole principle two of the partial sums must be equal, that is, there exist $0\le l_1<l_2\le l$ such that 
$\sum_{l_1< i\le l_2} R_G(e_i) = \mathbf{0} + L_{\calP_0'}^{\mu}(H)$. So we can delete them from the equation.
We can repeat this process until there are at most $d'-1$ edges.
\end{proof}

So we have $\sum_{i\in [d'']} \bfi_{\calP_0'}(e_i)+\bfi_{\calP_0'}(S_0)\in L_{\calP_0'}^{\mu}(H)$. Let $D:=\bigcup_{i\in [d'']}e_i\cup S_0$ satisfying that $k\le |D|=k d''+k\le k(d'-1)+k\le k^2$.
At last, we switch to $(\calP_0, L_{\calP_0}^{\mu}(H))$ again and finish the perfect matching by absorption. Since $\bfi_{\calP_0'}(D)\in L_{\calP_0'}^{\mu}(H)$, by Lemma \ref{lem:PL} (iii), we have $\bfi_{\calP_0}(D)\in L_{\calP_0}^{\mu}(H)$.
Thus, we have the following equation
\[
\bfi_{\calP_0}(D) = \sum_{\bfv\in I_{\calP_0}^{\mu}(H)} a_{\bfv} \bfv,
\]
where $a_{\bfv}\in \mathbb{Z}$ for all $\bfv\in I_{\calP_0}^{\mu}(H)$. Since the equation above is a special case of equation \eqref{eq:Cmax}, we have $|a_{\bfv}|\le C$ for all $\bfv\in I_{\calP_0}^{\mu}(H)$.
Noticing that $a_{\bfv}$ may be negative, we can assume $a_{\bfv}=b_{\bfv} - c_{\bfv}$ such that one of $b_{\bfv}, c_{\bfv}$ is a nonnegative integer and the other is zero for all $\bfv\in I_{\calP_0}^{\mu}(H)$. So, we have
\[
\sum_{\bfv\in I_{\calP_0}^{\mu}(H)} c_{\bfv} \bfv + \bfi_{\calP_0}(D) = \sum_{\bfv\in I_{\calP_0}^{\mu}(H)} b_{\bfv} \bfv.
\]
This equation means that given any family $\F$ consisting of disjoint $\sum_{\bfv}c_{\bfv}$ $k$-sets $e_1^{\bfv},\dots, e_{c_{\bfv}}^{\bfv}\subseteq V(H)\setminus D$ for $\bfv\in I_{\calP_0}^{\mu}(H)$ such that $\bfi_{\calP_0}(e_{i}^{\bfv})=\bfv$ for all $i\in [c_{\bfv}]$, we can regard $V(\F)\cup D$ as the union of $b_{\bfv}$ $k$-sets $S_1^{\bfv},\dots, S_{b_{\bfv}}^{\bfv}$ such that  $\bfi_{\calP_0}(S_j^{\bfv})=\bfv$, $j\in [b_{\bfv}]$ for all $\bfv\in I_{\calP_0}^{\mu}(H)$. 
Since $c_{\bfv}\le C$ for all $\bfv$ and $V(M_2)\cap D=\emptyset$, we can pick the family $\F$ as a subset of $M_2$.
In summary, starting with the matching $M_0\cup M_1\cup M_2\cup M_3$ leaving $S_0$ unmatched, we delete the edges $e_1, \dots, e_{d''}$ from $M_{0}\cup M_3$ given by Claim~\ref{clm:35} and then leave $D=\bigcup_{i\in [d'']}e_i\cup S_0$ unmatched.
Then we delete the family $\F$ of edges from $M_2$ and leave $V(\F)\cup D$ unmatched.
Finally, we regard $V(\F)\cup D$ as the union of at most $\binom{k+d'-1}{d'}C+k\le k^kC$ $k$-sets $S$ with $\bfi_{\calP_0}(S)\in I_{\calP_0}^{\mu}(H)$. 

Note that by definition, $e_1, \dots, e_{d''}$ may intersect at most $d''\le k-1$ absorbing sets in $\F_{0}$, which cannot be used to absorb the sets obtained above.
Since each $k$-set $S$ has at least $2k^k C - k(k-1)>k^k C+k-1$ absorbing sets in $\F_{0}$, we absorb them by $\F_{0}$ greedily and get a perfect matching of $H$.
\end{proof}

\subsection{Proof of the Absorbing Lemma}

\begin{claim}\label{clm:abs}
Suppose $V_i$ is $(\beta, t)$-closed in $H$ for all $i\in [d]$.
Then any $k$-set $S$ with $\bfi_{\calP_0}(S)\in I_{\calP_0}^{\mu}(H)$ has at least $\frac{\mu \beta^k}{2^{k+1}} n^{t k^2}$ absorbing $t k^2$-sets.
\end{claim}

\begin{proof}
For a $k$-set $S=\{y_1,\dots, y_k\}$ with $\bfi_{\calP_0}(S)\in I_{\calP_0}^{\mu}(H)$, we construct absorbing $t k^2$-sets for $S$ as follows. We first fix an edge $e=\{x_1, \dots, x_k\}$ in $H$ such that $\bfi_{\calP_0}(e)=\bfi_{\calP_0}(S)\in I_{\calP_0}^{\mu}(H)$ and $e\cap S=\emptyset$. 
Note that we have at least $\mu n^k - k n^{k-1} >\frac{\mu}2 n^k$ choices for such $e$. 
Without loss of generality, we may assume that for all $i\in [k]$, $x_i, y_i$ are in the same partition set of $\calP_0$. Since $x_i$ and $y_i$ are $(\beta, t)$-reachable in $H$, there are at least $\beta n^{t k-1}$ $(t k-1)$-sets $T_i$ such that both $H[T_i\cup \{x_i\}]$ and $H[T_i\cup \{y_i\}]$ have perfect matchings. We pick disjoint reachable $(t k-1)$-sets for each $x_i, y_i$, $i\in [k]$ greedily, while avoiding the existing vertices. Since the number of existing vertices is at most $t k^2+k$, we have at least $\frac{\beta}2 n^{t k-1}$ choices for such $(t k-1)$-sets in each step.
Note that each of $e\cup T_1\cup \cdots \cup T_{k}$ is an absorbing set for $S$. First, it contains a perfect matching because each $T_i\cup \{x_i\}$ for $i\in [k]$ spans $t$ disjoint edges. Second, $H[e\cup T_1\cup \cdots \cup T_{k}\cup S]$ also contains a perfect matching because $e$ is an edge and each $T_i\cup \{y_i\}$ for $i\in [k]$ spans $t$ disjoint edges.  So we find at least $\frac{\mu \beta^k}{2^{k+1}} n^{t k^2}$ absorbing $t k^2$-sets for $S$.
\end{proof}

\begin{proof}[Proof of Lemma \ref{lem:abs}]
We pick a family $\mathcal{F}$ of $t k^2$-sets by including every $t k^2$-set with probability $p=c n^{-t k^2} \log n$ independently, uniformly at random. Then the expected number of elements in $\mathcal{F}$ is $p\binom{n}{t k^2}\le \frac{c}{tk^2}\log n$ and the expected number of intersecting pairs of $t k^2$-sets is at most
\[
p^2\binom{n}{t k^2} \cdot t k^2 \cdot \binom{n}{t k^2-1} \le \frac{c^2 (\log n)^2}{n}=o(1).
\]
Then by Markov's inequality, with probability $1-1/(t k^2)-o(1)$, $\mathcal{F}$ contains at most $c \log n$ sets and they are pairwise vertex-disjoint.

For every $k$-set $S$ with $\bfi_{\calP_0}(S)\in I_{\calP_0}^{\mu}(H)$, let $X_S$ be the number of absorbing sets for $S$ in $\mathcal{F}$. Then by Claim \ref{clm:abs},
\[
\mathbb{E}(X_S) \ge p \frac{\mu \beta^k}{2^{k+1}} n^{tk^2} = \frac{\mu \beta^k c\log n}{2^{k+1}}.
\]
By Chernoff's bound, 
\[
\mathbb{P}\left( X_S\le \frac12 \mathbb{E}(X_S) \right) \le \exp\left\{ -\frac18 \mathbb{E}(X_S) \right\} \le \exp\left\{ -\frac{\mu \beta^k c \log n}{2^{k+4}} \right\}= o(n^{-k})
\]
because $1/c \ll \{\beta,\mu\}$.
Thus, with probability $1-o(1)$, for each $k$-set $S$ with $\bfi_{\calP_0}(S)\in I_{\calP_0}^{\mu}(H)$, there are at least
\[
\frac12\mathbb{E}(X_S) \ge \frac{\mu \beta^k c\log n}{2^{k+2}} \gg 2k^k C(k)
\]
absorbing sets for $S$ in $\F$. We obtain $\F_{abs}$ by deleting the elements of $\mathcal{F}$ that are not absorbing sets for any $k$-set $S$ and thus $|\F_{abs}|\le |\F|\le c\log n$.
\end{proof}

\subsection{Proof of Lemma \ref{lem:PL}}

In this subsection we prove Lemma \ref{lem:PL}. Our main goal is to build a partition $\calP=\{V_1,\dots, V_d\}$ of $V(H)$ for some $d\le k$ such that every $V_i$ is $(\beta, 2^{k-1})$-closed in $H$ for some $\beta>0$. 
For any $v\in V(H)$, let $\tilde{N}_{\beta, i}(v)$ be the set of vertices in $V(H)$ that are $(\beta, i)$-reachable to $v$.

\begin{proposition}\label{prop:Nv}
Suppose $H$ is a $k$-graph on $n$ vertices satisfying $\delta_{k-1}(H)\ge (1/k-\r)n$. For $\a>0$ and any $v\in V(H)$, $|\tilde{N}_{\a, 1}(v)| \ge (1/k - \r - 2k!\a) n$.
\end{proposition}

\begin{proof}
First note that $\delta_{k-1}(H)\ge (1/k-\r) n$ implies that $\delta_{1}(H)\ge (1/k-\r) \binom{n-1}{k-1}$. Fix a vertex $v\in V(H)$, note that for any vertex $u$, $u\in \tilde{N}_{\a, 1}(v)$ if and only if $|N_H(u)\cap N_H(v)|\ge \a n^{k-1}$.
By double counting, we have
\[
 |N_H(v)|\delta_{k-1}(H) \le \sum_{S\in N_H(v)} \deg_H(S) < |\tilde{N}_{\a, 1}(v)|\cdot |N_H(v)|+n\cdot \a {n}^{k-1}.
\]
Thus, $|\tilde{N}_{\a, 1}(v)|> \delta_{k-1}(H) - \frac{\a n^k}{|N_H(v)|}\ge (1/k - \r-2k!\a) n$ as $|N_H(v)|\ge \delta_1(H)\ge (1/k-\r) \binom{n-1}{k-1}$.
\end{proof}

The following lemma provides the partition $\calP_0$ in Lemma \ref{lem:PL}. 

\begin{lemma}\label{lem:P}
Given $0<\a \ll \delta, \delta'$, there exists constant $\beta>0$ satisfying the following. Assume an $n$-vertex $k$-graph $H$ satisfies that $|\tilde{N}_{\a, 1}(v)| \ge \delta' n$ for any $v\in V(H)$ and $\delta_1(H)\ge \delta\binom{n-1}{k-1}$. Then we can find a partition $\calP_0$ of $V(H)$ into $V_1,\dots, V_d$ with $d\le \min\{\lfloor 1/\delta \rfloor, \lfloor 1/\delta' \rfloor\}$ such that for any $i\in [d]$, $|V_i|\ge (\delta' - \a) n$ and $V_i$ is $(\beta, 2^{\lfloor 1/\delta \rfloor-1})$-closed in $H$, in time $O(n^{2^{c-1}k+1})$.
\end{lemma}

We will use the following simple result from \cite{LM1} to prove Lemma \ref{lem:P}.

\begin{proposition}\cite[Proposition 2.1]{LM1}\label{prop21}
For $\e, \beta>0$ and integer $i\ge 1$, there exists $\beta_0>0$ and an integer $n_0$ satisfying the following. Suppose $H$ is a $k$-graph of order $n\ge n_0$ and there exists a vertex $x\in V(H)$ with $|\tilde{N}_{\beta, i}(x)|\ge \e n$. Then for all $0<\beta'\le \beta_0$, $\tilde{N}_{\beta,i}(x)\subseteq \tilde{N}_{\beta',i+1}(x)$.
\end{proposition}

\medskip
\begin{proof}[Proof of Lemma \ref{lem:P}]
Let $c=\lfloor 1/\delta \rfloor$ (then $(c+1)\delta -1>0$) and $\e=\a/c$.
We choose constants satisfying the following hierarchy
\[
1/n\ll \beta = \beta_{c-1}\ll \beta_{c-2}\ll \cdots \ll \beta_1 \ll \beta_0 \ll \e, (c+1)\delta-1.
\]

Throughout this proof, given $v\in V(H)$ and $i\in [c-1]$, we write $\tilde{N}_{\beta_i, 2^i}(v)$ as $\tilde N_{i}(v)$ for short. 
Note that for any $v\in V(H)$, $|\tilde N_0(v)|=|\tilde{N}_{\beta_0, 1}(v)|\ge |\tilde{N}_{\a, 1}(v)| \ge \delta' n$ because $\beta_0<\a$.
We also say $2^i$-reachable (or $2^i$-closed) for $(\beta_i, 2^i)$-reachable (or $(\beta_i, 2^i)$-closed). By Proposition \ref{prop21} and the choice of $\beta_i$'s, we may assume that $\tilde {N}_{i}(v)\subseteq \tilde {N}_{i+1}(v)$ for all $0\le i<c-1$ and all $v\in V(H)$.
Hence, if $W\subseteq V(H)$ is $2^i$-closed in $H$ for some $i\le c-1$, then $W$ is $ 2^{c-1}$-closed.

Recall that two vertices $u$ and $v$ are 1-reachable in $H$ if $|N_H(u)\cap N_H(v)|\ge \beta_0 n^{k-1}$. We first note that any set of $c+1$ vertices in $V(H)$ contains two vertices that are 1-reachable in $H$ because $\delta_1(H)\ge \delta\binom{n-1}{k-1}$ and $(c+1)\delta-1\ge \binom{c+1}2\beta_0$. 
Also we can assume that there are two vertices that are not $2^{c-1}$-reachable to each other, as otherwise $V(H)$ is $2^{c-1}$-closed and we get a trivial partition $\calP_0=\{V(H)\}$.

Let $d$ be the largest integer such that
there exist $v_1,\dots, v_{d}\in V(H)$ such that no pair of them are $2^{c+1-d}$-reachable in $H$. 
Note that $d$ exists by our assumption and $2\le d\le c=\lfloor 1/\delta \rfloor$ by our observation. 
Fix such $v_1,\dots, v_{d}\in V(H)$, by Proposition \ref{prop21}, we can assume that any two of them are not $2^{c-d}$-reachable in $H$. 
Consider $\tilde N_{c-d}(v_i)$ for all $i\in [d]$ and we have the following facts.
\begin{enumerate}[(i)]
\item Any $v\in V(H)\setminus\{v_1,\dots, v_{d}\}$ must be in $\tilde N_{c-d}(v_i)$ for some $i\in [d]$, as otherwise $v, v_1,\dots, v_{d}$ contradicts the definition of $d$. 

\item $|\tilde N_{c-d}(v_i)\cap \tilde N_{c-d}(v_j)|<\e n$ because $v_i, v_j$ are not $2^{c+1-d}$-reachable in $H$.
Indeed, otherwise we get at least
\[
\frac{\e n}{(2^{c+1-d}k-1)!} (\beta_{c-d}n^{2^{c-d}k-1} - n^{2^{c-d}k-2}) (\beta_{c-d}n^{2^{c-d}k-1} - 2^{c-d}k n^{2^{c-d}k-2})  \ge \beta_{c+1-d} n^{2^{c+1-d}k-1}
\]
reachable $(2^{c+1-d}k-1)$-sets for $v_i, v_j$, which means that they are $2^{c+1-d}$-reachable, a contradiction.
Note that we get the lower bound of the number of the reachable sets for $v_i, v_j$ above by fixing one element $w\in \tilde N_{c-d}(v_i)\cap \tilde N_{c-d}(v_j)$, one $2^{c-d}$-reachable set $S$ for $v_i$ and $w$ (not containing $v_j$), and then one $2^{c-d}$-reachable set for $v_j$ and $w$ (not intersecting $\{v_i\}\cup S$).
Finally, it is divided by $(2^{c+1-d}k-1)!$ to eliminate the effect of overcounting.
\end{enumerate}
Note that (ii) and $|\tilde N_{c-d}(v_i)|\ge |\tilde{N}_{0}(v_i)| \ge \delta' n$ for $i\in [d]$ imply $d\delta' n - \binom d2 \e n\le n$. So we have $d\le (1+d^2 \e)/\delta'$. Since $\e \le \a \ll \delta'$, we have $d\le \lfloor 1/\delta' \rfloor$ and thus, $d\le \min\{\lfloor 1/\delta \rfloor, \lfloor 1/\delta' \rfloor\}$.

For $i\in [d]$, let $U_i=(\tilde N_{c-d}(v_i)\cup \{v_i\})\setminus \bigcup_{j\in [d]\setminus \{i\}} \tilde N_{c-d}(v_j)$. Note that for $i\in [d]$, $U_i$ is $2^{c-d}$-closed in $H$. Indeed, if there exist $u_1, u_2\in U_i$ that are not $2^{c-d}$-reachable in $H$, then $\{u_1, u_2\}\cup (\{v_1,\dots, v_{d}\}\setminus\{v_{i}\})$ contradicts the definition of $d$.

Let $U_0=V(H)\setminus (U_1\cup\cdots \cup U_{d})$. By (i) and (ii), we have $|U_0|\le \binom{d}{2}\e n$. We will move vertices of $U_0$ greedily to $U_i$ for some $i\in [d]$. For any $v\in U_0$, since $|\tilde N_0(v)\setminus U_0|\ge \delta' n - |U_0|\ge d\e n$, there exists $i\in [d]$ such that $v$ is 1-reachable to at least $\e n$ vertices in $U_i$. In this case we add $v$ to $U_i$ (we add $v$ to an arbitrary $U_i$ if there are more than one such $i$). 
Let the resulting partition of $V(H)$ be $V_1,\dots, V_{d}$. 
Note that we have $|V_i|\ge |U_i|\ge |\tilde N_{c-d}(v_i)| - d \e n\ge |\tilde N_0 (v_i)| - c\e n \ge (\delta' - \a )n$. 
Observe that in each $V_i$, the `farthest' possible pairs are those two vertices both from $U_0$, which are $2^{c-d+1}$-reachable in $H$. 
Thus, each $V_i$ is $2^{c-d+1}$-closed in $H$, so $2^{c-1}$-closed in $H$ because $d\ge 2$. 

\medskip

We estimate the running time as follows.
First, for every two vertices $u, v\in V(H)$, we determine if they are $2^{i}$-reachable for $0\le i\le c-1$. This can be done by testing if any $(2^i k -1)$-set $S\in \binom{V(H)\setminus \{u,v\}}{2^i k -1}$ is a reachable set for $u$ and $v$, namely, if both $H[S\cup \{u\}]$ and $H[S\cup \{v\}]$ have perfect matchings or not, which can be checked by listing every set of $2^{i}$ edges on them, in constant time. If there are at least $\beta_{i} n^{2^{i}k-1}$ reachable $(2^{i}k-1)$-sets for $v_i$ and $v_j$, then they are $2^{i}$-reachable. Since we need time $O(n^{2^{c-1}k-1})$ to list all $2^{c-1}k-1$ sets for all pairs $u, v$ of vertices, this can be done in time $O(n^{2^{c-1}k+1})$. 
Second, we search the set of vertices $v_1,\dots, v_d$ such that no pair of them are $2^{c+1-d}$-reachable for all $2\le d\le c$. With the reachability information at hand, this can be done in time $O(n^{c})$.
We then fix the largest $d$ as in the proof. If such $d$ does not exist, then we get $\calP_0=\{V(H)\}$ and output $\calP_0$. Otherwise, we fix any $d$-set $v_1,\dots, v_d$ such that no pair of them are $2^{c+1-d}$-reachable. We find the partition $\{U_0, U_1, \dots, U_d\}$ by identifying $\tilde{N}_{c-d}(v_i)$ for $i\in [d]$, in time $O(n)$.
Finally we move vertices in $U_0$ to $U_1,\dots, U_d$, depending on $|\tilde{N}_0(v)\cap U_i|$ for $v\in U_0$ and $i\in [d]$, which can be done in time $O(n^2)$. Thus, the running time for finding $\calP_0$ is $O(n^{2^{c-1}k+1})$.
\end{proof}

\medskip
\begin{proof}[Proof of Lemma \ref{lem:PL}]
Fix $0<\r\ll 1/k$. We apply Lemma \ref{lem:P} with $\a\ll \r$, $\delta=1/k-\r$, and $\delta'=1/k-\r-2k!\a$ and get $\beta>0$.
Suppose
\[
1/n\ll\{\beta, \mu_0\} \ll \r, 2^{-k}.
\]
Let $H$ be a $k$-graph on $n$ vertices satisfying $\delta_{k-1}(H)\ge (1/k-\r)n$. By Proposition \ref{prop:Nv}, for any $v\in V(H)$, $\tilde{N}_{\a, 1}(v) \ge (1/k - \r - 2k!\a) n=\delta' n$.
Since we also have $\delta_1(H)\ge \delta\binom{n-1}{k-1}$, we apply Lemma \ref{lem:P} on $H$ and get a partition $\calP_0=\{V_1, \dots, V_d\}$ of $V(H)$ in time $O(n^{2^{k-1}k+1})$. 
Note that $|V_i|\ge (\delta'-\a)n \ge (1/k-2\r)n$ for all $i\in [d]$ because $\a\ll \r$. Also we know that $d\le \lfloor 1/\delta\rfloor=k$ and each $V_i$ is $(\beta, 2^{k-1})$-closed.

Let $K=(k+1)^{d-1}$. We claim that we can pick a constant $\mu$ such that $K^{-\binom{k+d-1}{k}}\mu_0\le \mu\le \mu_0$ and
\begin{equation}\label{eq:I}
L_{\calP_0}^\mu(H)=L_{\calP_0}^{\mu/K}(H).
\end{equation}
Indeed, it suffices to pick such a $\mu$ so that $I_{\calP_0}^\mu(H)=I_{\calP_0}^{\mu/K}(H)$. This means that we will not `witness' more vectors even if we loosen our selection parameter $\mu$ by a factor $K$. 
Note that $|L_{\max}^d|=\binom{k+d-1}{k}$.
So if $I_{\calP_0}^{\mu_0}(H)\neq I_{\calP_0}^{\mu_0/K}(H)$, we pick $\mu_0/K$ as the new candidate, check it and repeat until we get the desired $\mu$. Note that in each intermediate step for some $\mu'$, we witness at least one new vector in $I_{\calP_0}^{\mu'/K}(H)$. So the process will terminate in at most $\binom{k+d-1}{k}$ steps and the resulting value $\mu$ satisfying $\mu\ge K^{-\binom{k+d-1}{k}}\mu_0$. Note that we find $\mu$ in constant time and we have the same hierarchy of constants after replacing $\mu_0$ by $\mu$.

It is possible that $(\calP_0, L_{\calP_0}^{\mu}(H))$ contains transferrals. We merge $V_i$ and $V_j$ into one vertex set if the transferral $\bfu_i - \bfu_j$ appears in $L_{\calP_0}^{\mu}(H)$ and repeat until there is no transferral in the resulting pair, denoted by $(\calP_0', L_{\calP_0'}^{\mu}(H))$, where $\calP_0'=\{V_1', \dots, V_{d'}'\}$ for some $d'\le d\le k$. Note that we get $\calP_0'$ from $\calP_0$ in time $O(n^{k})$. Indeed, we merge parts at most $d-1$ times and in each step, we identify the set of robust edge vectors by visiting all edges of $H$ and then determine if any transferral appears in the lattice in constant time.
Thus, overall, we find the pair $(\calP_0', L_{\calP_0'}^{\mu}(H))$ in time $O(n^{k})$. 

\begin{claim}\label{clm:vector}
Fix $\mu>0$. Given a partition $\calP_1=\{V_1,\dots, V_{|\calP_1|}\}$ such that $\bfu_1 - \bfu_2\in L_{\calP_1}^{\mu}(H)$ and let $\calP_1'$ be the partition obtained from merging $V_1, V_2$ of $\calP_1$. Then for any $D\subseteq V(H)$ such that $\bfi_{\calP_1'}(D)\in L_{\calP_1'}^\mu(H)$, we have $\bfi_{\calP_1}(D)\in L_{\calP_1}^{\mu/(k+1)}(H)$.
\end{claim}

\begin{proof}
For any vector $\bfv$ with respect to $\calP_1$, let $\bfv|_{\calP_1'}$ be the projection of $\bfv$ on $\calP_1'$, which is a vector with respect to $\calP_1'$.
Let $D\subseteq V(H)$ be any vertex set such that $\bfi_{\calP_1'}(D)\in L_{\calP_1'}^\mu(H)$.
So we have the equation $\bfi_{\calP_1'}(D) = \sum_{\bfv'\in I_{\calP_1'}^{\mu}(H)} a_{\bfv'} \bfv'$, where $a_{\bfv'}\in \mathbb{Z}$ for all $\bfv'\in I_{\calP_1'}^{\mu}(H)$. 
Note that for each $\bfv'\in I_{\calP_1'}^\mu(H)$, there exist at most $k+1$ vectors $\bfv_i\in L_{\max}^{|\calP_1|}$ such that $\bfv_i |_{\calP_1'}=\bfv'$. Thus, by the pigeonhole principle, there exists $\bfv\in I_{\calP_1}^{\mu/(k+1)}(H)$ such that $\bfv|_{\calP_1'}=\bfv'$.
Let $\bfi_0=\sum_{\bfv'\in I_{\calP_1'}^{\mu}(H)} a_{\bfv'} \bfv$, which is a $|D|$-vector in $L_{\calP_1}^{\mu/(k+1)}(H)$. Note that $\bfi_{\calP_1}(D)|_{\calP_1'} = \bfi_{\calP_1'}(D) = \bfi_0|_{\calP_1'}$. This implies that $\bfi_{\calP_1}(D) = \bfi_0$ or $\bfi_{\calP_1}(D) - \bfi_0$ equals a multiple of $\bfu_1 - \bfu_2$. Since $\bfu_1 - \bfu_2\in L_{\calP_1}^{\mu}(H)$, we have $\bfi_{\calP_1}(D) - \bfi_0\in L_{\calP_1}^{\mu}(H)$ and thus $\bfi_{\calP_1}(D)=\bfi_{\calP_1}(D) - \bfi_0+\bfi_0\in L_{\calP_1}^{\mu/(k+1)}(H)$.
\end{proof}

Now let us show Lemma \ref{lem:PL} (iii). Fix any $D\subseteq V(H)$ such that $\bfi_{\calP_0'}(D)\in L_{\calP_0'}^\mu(H)$. We apply Claim \ref{clm:vector} $d-d'$ times and get that $\bfi_{\calP_0}(D)\in L_{\calP_0}^{\mu/(k+1)^{d-d'}}(H)$. Since $\mu/K\le {\mu/(k+1)^{d-d'}}\le \mu$, by~\eqref{eq:I}, we get $\bfi_{\calP_0}(D)\in L_{\calP_0}^{\mu/(k+1)^{d-d'}}(H)=L_{\calP_0}^{\mu/K}(H)=L_{\calP_0}^{\mu}(H)$.

\medskip 

It remains to show that $(\calP_0', L_{\calP_0'}^{\mu}(H))$ is a full pair for $H$.
Indeed, since $(\calP_0', L_{\calP_0'}^{\mu}(H))$ is transferral-free, it remains to show that $I_{\calP_0'}^\mu(H)$ is full. Assume to the contrary, that there exists a $(k-1)$-vector $\bfv$ such that $\bfv+\bfu_i \notin I_{\calP_0'}^\mu(H)$ for all $i\in [d']$.
Note that since $\bfv+\bfu_i\notin I_{\calP_0'}^\mu(H)$, there are less than $\mu n^k$ edges $e$ in $H$ with $\bfi_{\calP_0'}(e)=\mathbf{v}+\bfu_{i}$. So there are less than $d' \mu n^k$ edges that contain some $(k-1)$-set with index vector $\bfv$. But since there are at least $\binom{\min_{j\in [d']} |V_j'|}{k-1}$ $(k-1)$-sets with index vector $\mathbf{v}$ and $\delta_{k-1}(H)\ge n/k-\r n$, the number of such edges is at least $\frac1k\left(\frac{n}{k}-\r n\right) \binom{\min_{j\in [d']} |V_j'|}{k-1}\ge \frac1k \left(\frac{n}{k}-\r n\right) \binom{n/k-2\r n}{k-1}>d'\mu n^k$, a contradiction.
\end{proof}

\section{The Extremal Case}

Our goal of this section is to prove Theorem \ref{thm3}.
We remark that the $k$-graphs in Construction \ref{con12} do not appear in our proof because they achieve smaller minimum codegrees than those $k$-graphs in Construction \ref{con13} if $k$ is even and Construction \ref{con12} and Construction \ref{con13} are the same if $k$ is odd.

We use the following result of Pikhurko \cite{Pik}, stated here in a less general form.

\begin{theorem}\cite[Theorem 3]{Pik}\label{thm:pik}
Let $H$ be a $k$-partite $k$-graph with the $k$-partition $V(H)=V_1\cup V_2\cup \cdots \cup V_k$ such that $|V_i|=m$ for all $i\in [k]$. Let $\delta_{\{1\}}(H)=\min \{ |N(v_1)|: v_1\in V_1 \}$ and
\[
\delta_{[k]\setminus \{1\}}(H) = \min\{ |N(v_2,\dots, v_k)|: v_i\in V_i \text{ for every }2\le i\le k \}.
\]
For sufficiently large integer $m$, if
\[
\delta_{\{1\}}(H) m + \delta_{[k]\setminus\{1\}}(H) m^{k-1} \ge \frac32 m^k,
\]
then $H$ contains a perfect matching.
\end{theorem}

\subsection{Preliminary and the proof of Theorem \ref{thm3}}

Fix a sufficiently small $\e>0$. Let $n$ be a sufficiently large integer.
Suppose $H$ is a $k$-graph on $n$ vertices such that $\delta_{k-1}(H)\ge n/k$ and $H\notin \mathcal{H}_{n,k}$. Assume that $H$ is $\e$-extremal, namely, there is an independent subset $S\subseteq V(H)$ with $|S| \ge (1-\e)\frac{k-1}k n$. 
Let $\a=\e^{1/3}$. We partition $V(H)$ as follows. Let $C$ be a maximum independent subset of $V(H)$. Define
\begin{equation}\label{eq:A}
A=\left\{x\in V\setminus C: \deg(x, C)\ge (1-\a) \binom{|C|}{k-1}\right\},
\end{equation}
and $B=V(H)\setminus (A\cup C)$. We first observe the following bounds of $|A|, |B|, |C|$.

\begin{claim}\label{clm:size}
$|A|\ge n/k-\a^2 n$, $|B|\le \a^2 n$, and $(1-{\e})\frac {(k-1)n}k\le |C|\le \frac {(k-1)n}k$.
\end{claim}

\begin{proof}
The lower bound for $|C|$ follows from our hypothesis immediately. For any $S\subseteq C$ of order $k-1$, we have $N(S)\subseteq A\cup B$. 
By the minimum degree condition, we have 
\begin{equation}\label{eq:ab}
\frac{n}k \le |N(S)| \le |A|+|B| =n-|C| \le \frac nk + \e \frac {(k-1)n}k,
\end{equation}
which gives the upper bound for $|C|$.
By the definitions of $A$ and $B$, we have
\[
\frac{n}k \binom{|C|}{k-1} \le e((A\cup B)C^{k-1})\le (1-\a)\binom{|C|}{k-1} |B| + \binom{|C|}{k-1} |A|,
\]
where $e((A\cup B)C^{k-1})$ denotes the number of edges that contain $k-1$ vertices in $C$ and one vertex in $A\cup B$.
Thus, we get $n/k \le |A|+|B|-\a |B|$, which gives that $\a |B| \le |A|+|B| - n/k\le \e n$ by \eqref{eq:ab}.
So $|B|\le \a^2 n$ and by \eqref{eq:ab} again, $|A|\ge n/k-|B|\ge n/k-\a^2 n$. 
\end{proof}

The partition which we will work on in this section is $\calP=(A\cup B, C)$. For $0\le i\le k$, we say an edges $e$ is \emph{an $i$-edge} if $|e\cap (A\cup B)|=i$. 
We remark that as mentioned before, since $H$ is close to the space barrier, it is rather `fragile' -- even the bad choice of one edge may lead the remaining $k$-graph into the space barrier, so we cannot use the robust edge-lattice and apply the discussions in Section 3.

Let us list our auxiliary lemmas.

\begin{lemma}\label{lem:main}
Fix any even $2\le i\le k$. Assume that $|A\cup B|\ge n/k+i-1$ and $H$ contains no $j$-edge for all even $0\le j\le i-2$. If $H$ contains an $i$-edge, then $H$ contains a perfect matching.
\end{lemma}

\begin{lemma}\label{lem:exact}
Fix any even $0\le i\le k$. If $|A\cup B|=n/k+i$ and $H$ contains no $j$-edge for all even $0\le j\le i$, then $H$ contains a perfect matching.
\end{lemma}

\begin{lemma}\label{lem:last}
If $H$ contains no $j$-edge for all even $0\le j\le k$ and $H\notin \mathcal{H}_{n,k}$, then $H$ contains a perfect matching.
\end{lemma}

We postpone the proofs of these lemmas to the following subsections and prove Theorem \ref{thm3} first.

\begin{proof}[Proof of Theorem \ref{thm3}]
The proof of Theorem \ref{thm3} runs in an algorithmic way as follows. The case when $|A\cup B|=n/k$ is covered by Lemma \ref{lem:exact} with $i=0$. Next by Lemma \ref{lem:main}, if $|A\cup B|\ge n/k+1$ and there is a 2-edge in $H$, then $H$ contains a perfect matching. So we may assume that $H$ contains no 2-edge. Consider any $(k-1)$-set $S$ with $|S\cap (A\cup B)|=2$, since there is no 2-edge, we get $N(S)\subseteq A\cup B$ and thus $|A\cup B|\ge n/k+2$. By Lemma \ref{lem:exact} again, if $|A\cup B|=n/k+2$ and $H$ contains no 2-edge, then $H$ contains a perfect matching. So we can assume that $|A\cup B|\ge n/k+3$ and $H$ contains no 2-edge. If $H$ contains one 4-edge, then by Lemma \ref{lem:main}, $H$ has a perfect matching. After $\lfloor k/2\rfloor$ iterations, we can assume that $H$ contains no $j$-edge for all even $0\le j\le k$. In this case, by Lemma \ref{lem:last}, we find a perfect matching provided that $H\notin \mathcal{H}_{n,k}$.
\end{proof}

\subsection{Proof of Lemma \ref{lem:main}}

Fix any even $2\le i\le k$. Assume that $|A\cup B|\ge n/k+i-1$ and $H$ contains no $j$-edge for all even $0\le j\le i-2$. Assume that $H$ contains an $i$-edge. 

Let us first outline our proof. Our main goal is to remove a small matching $M$ that covers every vertex in $B$ such that the sets of remaining vertices $A\setminus V(M)$ and $C\setminus V(M)$ satisfy $|C\setminus V(M)|=(k-1)|A\setminus V(M)|$. Then we partition $C\setminus V(M)$ into $k-1$ parts and apply Theorem \ref{thm:pik} and get a perfect matching on $V(H)\setminus V(M)$. So we get a perfect matching of $H$.

Roughly speaking, since $|\calP|=2$, the `divisibility' is reduced to `parity', which means that if we need to `repair' the divisibility, one edge is enough. An $i$-edge $e_0$ will be such edge for repairing -- we will add $e_0$ to our matching at the very beginning of our proof. But the divisibility barrier may not appear, in which case, choosing $e_0$ makes the parity bad. However, we cannot foresee this at the beginning. So at some intermediate step, if we find out that we made the wrong decision, we just free $e_0$ from our partial matching and the parity will be good again (in this case, the parity was good at the beginning).

Now we start our proof. We separate two cases.

\medskip
\noindent {\bf Case 1. }$i=2$ and there is a 2-edge $e_0$ such that $|e_0 \cap A|=|e_0 \cap B|=1$.
\medskip

Let $x=e_0\cap B$. Since $C$ is a maximum independent set, there exists a $(k-1)$-set $S_x\subseteq C$ such that $e_x:=\{x\}\cup S_x\in E(H)$. Note that $S_x\setminus e_0$ may intersect $e_0\cap C$. We reserve $S_x$ for future use, which means, we will not use its vertices later until the very last step.

We will build four disjoint matchings $M_1$, $M_2$, $M_3$, and $M_4$ in $H$, whose union gives the desired perfect matching in $H$. For $i\in [3]$, let $A_i=A\setminus V(\cup_{j\in [i]} M_j)$ and $C_i=C\setminus V(\cup_{j\in [i]} M_j)$ be the sets of uncovered vertices of $A$ and $C$, respectively. Let $n_i=|V(H)\setminus V(\cup_{j\in [i]} M_j)|$.

\bigskip
\noindent \emph{Step 1. Small matchings $M_1$ and $M_2$ covering $B$.}

Let $t:= n/k - |A|$. We let $M_1=\{e_0\}$ if $t\le 0$. Otherwise, we build the first matching $M_1$ of size $t+1$ as follows. By Claim \ref{clm:size}, we know that $t=n/k - |A|\le \a^2 n$. 
By $\delta_{k-1}(H)\ge n/k$ and the definition of $t$, we have $\delta_{k-1}(H[B\cup C])\ge t$. Since $|C|\le \frac {(k-1)n}k-1$, we have $|B|= n - |C| - |A|\ge n/k - |A|+1=t+1$. 

We claim that we can find a matching of $t$ 1-edges in $(B\cup C)\setminus (e_0\cup S_x)$. Let $M_1$ be the union of these edges and $e_0$.
Indeed, we pick $t$ arbitrary disjoint $(k-1)$-sets $S_1,\dots, S_{t}$ from $C\setminus (e_0\cup S_x)$. Since $C$ is an independent set, each of $S_i$ has at least $t-1$ neighbors in $B\setminus x$ for $i\in [t]$. 
Consider the bipartite graph between $B\setminus x$ and $\{S_1,\dots, S_{t}\}$, in which we put an edge if $\{v\}\cup S_i\in E(H)$ for $v\in B\setminus x$ and $i\in [t]$.
By the K\"onig-Egervary Theorem, either we have a matching of size $t$ (then we are done), or there is a vertex cover of order $t-1$. Since the degree of any $S_1,\dots, S_t$ is at least $t-1$ in the auxiliary bipartite graph, the vertex cover must be in $B\setminus x$, denoted by $B'$ (of order $t-1$), 
and every vertex in $B'$ is adjacent to all $S_i$ for $i\in [t]$.
Now consider $(k-1)$-sets in $C\setminus (\bigcup_{i\in [t]}S_i \cup e_0\cup S_x)$. If our claim does not hold, namely, there is no $t$ disjoint 1-edges, then all these $(k-1)$-sets are adjacent to all vertices in $B'$. 
Note that $|C\setminus (\bigcup_{i\in [t]}S_i \cup e_0\cup S_x)|\ge |C|-(k-1)t -2k\ge (1-2k\a^2)|C|$, because $t\le \a^2 n\le 2\a^2 |C|$.
So for any $v\in B'$, we have
\[
\deg(v, C)\ge \binom{(1-2k\a^2)|C|}{k-1} \ge ((1-2k\a^2)^{k-1}-o(1))\binom{|C|}{k-1} > (1-\a)\binom{|C|}{k-1},
\]
as $\a$ is small enough. This contradicts the fact that $v\notin A$. So the claim holds.

Next we build the second matching $M_2$ that covers all vertices in $B\setminus V(M_1)$. 
For each $v\in B\setminus V(M_1)$, we pick $k-2$ arbitrary vertices from $C\setminus S_x$ not covered by the existing matching, and an uncovered vertex in $V$ to complete an edge and add it to $M_2$. Since $\delta_{k-1}(H)\ge n/k$ and the number of vertices covered by the existing matching is at most $k |B|\le k\a^2 n<\delta_{k-1}(H)$, such edge always exists.

\medskip
Our construction guarantees that each edge in $M_1\cup M_2$ contains at least one vertex from $B$ and thus $|M_1\cup M_2|\le |B|$. 
We claim that $|A_1|\ge n_1/k$ and $|A_2|\ge n_2/k$.
To see the bound for $|A_1|$, we separate two cases depending on $t$. When $t>0$, by the definition of $M_1$, we have
\[
|A_1| = \frac{n}{k} - t -1 = \frac{n - k|M_1|}{k}=\frac{n_1}{k}.
\]
Otherwise $t\le 0$, we have $n_1=n-k$ and $|A_1|=|A|-1\ge n/k-1= n_1/k$. 
For the bound for $|A_2|$, since each edge of $M_2$ contains at most one vertex of $A$, we have
\[
|A_2| \ge |A_1| - |M_2| \ge \frac{n_1}{k}  - |M_2| = \frac{n_2}{k}.
\]

Let $s:=|A_2| - n_2/k\ge 0$. Since $n_2=n - k|M_1\cup M_2|\ge n - k|B|\ge n - k\a^2 n$ and $|C|\ge (1-{\e})\frac {(k-1)n}k$ (Claim \ref{clm:size}), we get
\[
s\le n - |C| - \frac{n-k\a^2 n}{k}\le \e\frac {(k-1)n}k + \a^2 n \le 2\a^2 n.
\]

\bigskip
\noindent \emph{Step 2. A small matching $M_3$.}

We will construct a matching $M_3$ of size at most $2\a^2 n$ on $A_2\cup (C_2\setminus S_x)$ such that $|A_3| - {n_3}/{k}\in \{0, -1\}$. 
To see that this is possible, at some intermediate step, denote by $n'$ as the number of uncovered vertices of $H$ and denote by $A', C'$ as the sets of uncovered vertices in $A, C\setminus S_x$, respectively. Let $c=|A'| - {n'}/{k}$.
If $c> 0$, then we arbitrarily pick two vertices from $A'$, $k-3$ vertices from $C'$ and one vertex from $A'\cup C'$ to form an edge. Note that we pick a 2-edge or a 3-edge in each step.
As a result, $c$ decreases by $1$ or $2$. 
The iteration stops when $c$ becomes 0 or $-1$ after at most $s\le 2\a^2 n$ steps.
Note that we can always form an edge in each step because the number of covered vertices is at most $k|B|+k\cdot 2\a^2 n\le 3k\a^2 n<\delta_{k-1}(H)$. So we get a matching $M_3$ of at most $2\a^2 n$ edges.

\bigskip
\noindent \emph{Step 3. The last matching $M_4$.}

Now we have two cases, $|A_3| - {n_3}/{k}= -1$ or 0. In the former case, we delete the edge $e_0$ from $M_1$ and add $e_x$ to $M_1$. Note that this is possible because $S_x\subseteq C_3$. Let the resulting sets of uncovered vertices be $A_3', C_3'$, respectively. Also let $n_3':=|A_3'|+|C_3'|=n_3$. So $|A_3'|=|A_3|+1$ and we have $|A_3'| - {n_3'}/{k}= 0$, that is, $|C_3'| = (k-1)|A_3'|$. In the latter case we let $A_3'=A_3$ and $C_3'=C_3$. We have $|C_3'| = (k-1)|A_3'|$ immediately.
By definition, we have
\[
|A_3'|\ge |A| - |M_1\cup M_2| - 3|M_3|\ge n/k - \a^2 n -\a^2 n - 6\a^2 n \ge n/k - 8\a^2 n,
\]
as $|M_1\cup M_2|\le |B|\le \a^2 n$ and $|M_3|\le 2\a^2 n$.

Let $m:=|A_3'|$. Next, we partition $C_3'$ arbitrarily into $k-1$ parts $C^1, C^2,\dots, C^{k-1}$ of the same size $m$. We want to apply Theorem \ref{thm:pik} on the $k$-partite $k$-graph $H':=H[A_3', C^1,\dots, C^{k-1}]$. Let us verify the assumptions. 
First, since $C_3'$ is independent, for any set of $k-1$ vertices $v_1,\dots, v_{k-1}$ such that $v_i\in C^i$ for $i\in [k-1]$, the number of its non-neighbors in $A\cup B$ is at most
\[
|A| + |B| - {n}/{k} \le n/k + \e \frac {(k-1)n}k - n/k \le k\e m,
\]
where we use \eqref{eq:ab} in the first inequality and the last inequality follows from $m=|A_3'|\ge n/k - 8\a^2 n>\frac{k-1}{k^2}n$.
So we have $\delta_{[k]\setminus\{1\}}(H')\ge m - k\e m= (1-k\e)m$.
Next, by \eqref{eq:A}, for any $v\in A_3'$, we have
\[
\overline{\deg}_{H}(v, C)\le \a \binom{|C|}{k-1}\le \a \frac{|C|^{k-1}}{(k-1)!}\le \a \frac{\left(\frac{k-1}{k}n \right)^{k-1}}{(k-1)!}\le \a \frac{(km)^{k-1}}{(k-1)!}= \a c_k m^{k-1},
\]
where $c_k=\frac{k^{k-1}}{(k-1)!}$.
This implies that $\delta_{\{1\}}(H') \ge (1-\a c_k) m^{k-1}$. Thus, we have
\[
\delta_{\{1\}}(H') m + \delta_{[k]\setminus\{1\}}(H') m^{k-1} \ge (1- \a c_k) m^{k-1} m + (1-k \e)m m^{k-1}>\frac32 m^k,
\]
as $\e$ is small enough. By Theorem \ref{thm:pik}, we find a perfect matching in $H'$, which gives the perfect matching $M_4$ on $A_3'\cup C_3'$. So $M_1\cup M_2\cup M_3\cup M_4$ gives a perfect matching of $H$.

\bigskip
\noindent {\bf Case 2. }$i=2$, there is a 2-edge $e_0$ and there is no 2-edge $e$ such that $|e \cap A|=|e \cap B|=1$; or $i$ is even with $4\le i\le k$ and there is an $i$-edge $e_0$.
\medskip

We first observe the following fact.

\begin{fact}\label{fact:BCC}
Assume that $H$ contains no 2-edge $e$ such that $|e\cap A|=|e\cap B|=1$, then for any $(k-1)$-tuple $S$ with $|S\cap B|=1$ and $|S\cap C|=k-2$, we have $\deg(S, C)\ge n/k-\a^2 n$.
\end{fact}

\begin{proof}
Since there is no such 2-edge, $N(S)\subseteq B\cup C$. By the minimum degree condition and $|B|\le \a^2 n$ by Claim \ref{clm:size}, we have $\deg(S, C)\ge n/k-\a^2 n$.
\end{proof}

Note that Fact \ref{fact:BCC} works under either assumption in this case. This simplifies Step 1 -- we only need to build one matching. But to be consistent with Case 1, we set $M_2=\emptyset$ in this case.

\bigskip
\noindent \emph{Step 1. A small matching $M_1$ covering $B$.}

We build $M_1$ as follows. First we add the $i$-edge $e_0$ to $M_1$. By Fact \ref{fact:BCC} and $|B|\le \a^2 n$, we greedily pick a matching $M'$ of $|B|$ 1-edges from $B\cup (C\setminus e_0)$.
Assume that $|e_0\cap B|=j\le i$. If $j>0$, denote the vertices by $x_1, \dots, x_j\in e_0\cap B$ and let $S_{x_1},\dots, S_{x_j}$ be the $(k-1)$-sets in $C$ that form edges $e_{x_1},\dots, e_{x_j}$ with $x_1,\dots, x_j$ in the matching $M'$, respectively. As in Case 1, we will reserve $S_{x_1}\setminus e_0, \dots, S_{x_j}\setminus e_0$ for future use.
If $j=0$, we add all edges of $M'$ to $M_1$. Otherwise, we add the $|B|-j$ edges of $M'$ that do not contain $x_1,\dots, x_j$ to $M_1$. So we have $|M_1|= |B|+1-j$.

We claim that $|A_1|\ge n_1/k$.
Recall that
\[
|A\cup B|\ge n/k+i-1=n_1/k+|M_1|+i-1=n_1/k+|B|+i-j.
\]
Since $|e_0\cap A|=i-j$, we have,
\[
|A_1| = |A|-(i-j) = |A\cup B| - |B|-(i-j) \ge n_1/k.
\]
Since $M_2=\emptyset$, we have $|A_2|\ge n_2/k$.

So $s:=|A_2| - n_2/k\ge 0$ and as in the previous case, $s\le 2\a^2 n$.

\bigskip
\noindent \emph{Step 2. A small matching $M_3$.}

We will construct a matching $M_3$ of 2-edges and 3-edges with size at most $2\a^2 n$ on $A_2\cup (C_2\setminus (S_{x_1}\cup\cdots\cup S_{x_j}))$ such that $|A_3| - {n_3}/{k}\in \{0, 1-i\}$. 
Similar as in Case 1, if we add a 2-edge (or a 3-edge) to $M_3$, then the value of $c$ decreases by 1 (or 2), respectively.
So if there is one 2-edge, we can construct $M_3$ of size at most $s$ such that $|A_3| - {n_3}/{k}=0$ (we can choose to include or exclude this 2-edge in $M_3$).
So if we cannot have $|A_3| - {n_3}/{k}=0$, then there is no 2-edge in $H[A_2\cup (C_2\setminus (S_{x_1}\cup\cdots\cup S_{x_j}))]$ and $s$ is odd.
In this case we add $(s+i-1)/2$ disjoint 3-edges to $M_3$ and therefore $|A_3| - {n_3}/{k}=1-i$.
Note that we always can form 2-edges or 3-edges similarly as in Case 1.
So we get a matching $M_3$ of at most $s\le 2\a^2 n$ edges.

\bigskip
\noindent \emph{Step 3. The last matching $M_4$.}

Now we have two cases, $|A_3| - {n_3}/{k}= 1-i$ or 0. In the former case, we delete the $i$-edge $e_0$ from $M_1$ and add the edges $e_{x_1},\dots, e_{x_j}$ to $M_1$ (if $j>0$). Let the resulting sets of uncovered vertices be $A_3', C_3'$, respectively. Also let $n_3':=|A_3'|+|C_3'|=n_3+k-j k$. So $|A_3'|=|A_3|+i-j$ and we have $|A_3'| - {n_3'}/{k}= 0$, namely, $|C_3'| = (k-1)|A_3'|$. In the latter case we let $A_3'=A_3$ and $C_3'=C_3$. We have $|C_3'| = (k-1)|A_3'|$ immediately.
By definition, we have
\[
|A_3'|\ge |A| - |M_1| - 3|M_3|\ge n/k - \a^2 n - (\a^2 n+1) - 6\a^2 n \ge n/k - 9\a^2 n,
\]
as $|M_1|\le |B|+1\le \a^2 n+1$ and $|M_3|\le 2\a^2 n$.

Let $m:=|A_3'|$. We partition $C_3'$ arbitrarily into $k-1$ parts $C^1, C^2,\dots, C^{k-1}$ of the same size $m$. We apply Theorem \ref{thm:pik} on the $k$-partite $k$-graph $H':=H[A_3', C^1,\dots, C^{k-1}]$ and get a perfect matching in $H'$, which gives the perfect matching $M_4$ on $A_3'\cup C_3'$. So $M_1\cup M_2\cup M_3\cup M_4$ gives a perfect matching of $H$. We omit the similar calculations.

\subsection{Proofs of Lemma \ref{lem:exact} and Lemma \ref{lem:last}}

\begin{proof}[Proof of Lemma \ref{lem:exact}]
Fix any even $0\le i\le k$. Assume that $|A\cup B|=n/k+i$ and $H$ contains no $j$-edge for all even $0\le j\le i$. If $i=0$, then we have $|A\cup B|=n/k$ and $|C|=\frac{k-1}{k}n$. By the minimum degree condition, every $k$-set containing exactly $k-1$ vertices in $C$ is an edge of $H$. Thus, we partition $V(H)$ into $n/k$ such $k$-sets and get a perfect matching of $H$.
So we may assume $i\ge 2$. 

Since there is no $i$-edge, we can take an $(i+1)$-edge $e_0$ such that $|e_0\cap A|=i+1$. Indeed, we take $i$ vertices from $A$ and $k-i-1$ vertices from $C$ and another vertex to form an edge. Since $H$ contains no $i$-edge and $|B|\le \a^2 n<n/k$, we can pick the last vertex from $A$ and get the desired $(i+1)$-edge $e_0$.

Next by Fact \ref{fact:BCC}, we find a matching of $|B|$ 1-edges that covers all vertices of $B$.
Let $A'$ and $C'$ be the set of uncovered vertices of $A$ and $C$, respectively. Note that we have $|A'|=n/k+i-|B|-(i+1)=n/k-|B|-1$ and
\[
|C'|= \frac{k-1}{k}n-i - (k-i-1) - (k-1)|B|=(k-1)|A'|.
\]
So as in the previous proofs, we partition $C'$ arbitrarily into $k-1$ parts, apply Theorem \ref{thm:pik} and get a perfect matching on $A'\cup C'$. Thus, we get a perfect matching of $H$.
\end{proof}

\begin{proof}[Proof of Lemma \ref{lem:last}]
Assume that $H$ contains no $j$-edge for all even $0\le j\le k$ and $H\notin \mathcal{H}_{n,k}$.
Since there is no 2-edge, by Fact \ref{fact:BCC}, we find a matching $M_1$ of $|B|$ 1-edges that covers all vertices of $B$. Let $C'$ be the set of uncovered vertices of $C$. 
Let $n'=|A|+|C'|$ and note that $n'/k=n/k-|B|$. 
Let
\[
s:=|A| - n'/k=|A|+|B|-n/k =|A\cup B| - n/k. 
\]
So $0\le s\le \e n$ by \eqref{eq:ab}. Moreover, we claim that $s$ is even. Indeed, since all edges of $H$ intersect $A\cup B$ in an odd number of vertices, if $s$ is odd, then $H\in \mathcal{H}_{n,k}$, a contradiction.
We greedily pick a matching $M_2$ of $s/2$ disjoint 3-edges, which is possible because $s\le \e n$ and $\delta_{k-1}(H)\ge n/k$. Let $A_2$ and $C_2$ be the set of vertices not covered by $M_1\cup M_2$. As in the previous proofs, we have $|C_2|=(k-1)|A_2|$. We partition $C_2$ arbitrarily into $k-1$ parts, apply Theorem \ref{thm:pik} and get a perfect matching $M_3$ on $A_2\cup C_2$. So we get a perfect matching $M_1\cup M_2\cup M_3$ of $H$.
\end{proof}

\section{Algorithms and the proof of Theorem \ref{thm:main}}

\subsection{A straightforward but slower algorithm}
Let $L_{odd}$ be the lattice generated by all two dimensional $k$-vectors with first coordinate odd, that is, $(1,k-1), (3, k-3),\dots, (k-1,1)$ if $k$ is even, and $(1,k-1), (3, k-3),\dots, (k,0)$ if $k$ is odd. It is easy to see that $L_{odd}$ is full. To check if a $k$-graph $H\in \mathcal{H}_{n,k}$, we find the bipartitions $\calP$ of $V(H)$ such that $\bfi_{\calP}(e)\in L_{odd}$ for every $e\in H$.
We use the algorithm Procedure ListPartitions in \cite{KKM13}. The following lemma \cite[Lemma 2.2]{KKM13} estimates the computation complexity of Procedure ListPartitions (although \cite[Lemma 2.2]{KKM13} was proved under the codegree condition $\delta_{k-1}(H)\ge n/k+\r n$, we can weaken the codegree condition as explained in \cite[Remark 2.3]{KKM13}).

\begin{lemma}\cite{KKM13}\label{lem22KKM}
Suppose $H$ is an $n$-vertex $k$-graph with $\delta_{k-1}(H)\ge n/k-2k(k-2)$. For any $d\in [k]$ and full edge-lattice $L\subseteq \mathbb{Z}^d$, there are at most $d^{2k-1}$ partitions $\calP$ of $V(H)$ such that $\bfi_{\calP}(e)\in L$ for every $e\in H$, and Procedure ListPartitions lists them in time $O(n^{k+1})$.
\end{lemma}

By Theorem \ref{thm:PM}, the straightforward way to determine the existence of a perfect matching is to check if $(\calP_0', L_{\calP_0'}^{\mu}(H))$ is soluble and if $H\notin \mathcal{H}_{n,k}$.

\begin{theorem}\label{thm:slower}
Fix $k\ge 3$. Let $H$ be an $n$-vertex $k$-graph with $\delta_{k-1}(H)\ge n/k$. Then there is an algorithm with running time $O(n^{2^{k-1}k+1})$, which determines whether $H$ contains a perfect matching.
\end{theorem}

\begin{proof}
Let $H$ be an $n$-vertex $k$-graph with $\delta_{k-1}(H)\ge n/k$. 
Note that it is trivial to determine the existence of a perfect matching if $n<n_0$ given by Theorem \ref{thm:PM}.
Our algorithm contains two parts when $n\ge n_0$. First we find the partition $\calP_0$ and $\calP_0'$ and check if $(\calP_0', L_{\calP_0'}^{\mu}(H))$ is soluble. Second, we check if $H\notin \mathcal{H}_{n,k}$. If both answers are `true', then $H$ contains a perfect matching by Theorem \ref{thm:PM}.

By Lemma \ref{lem:PL}, we find $\calP_0$ and $\calP_0'$ in time $O(n^{2^{k-1}k+1})$. To check the solubility, we check if $\bfi_{\calP_0'}(V(H)\setminus V(M))\in L_{\calP_0'}^{\mu}(H)$ for each matching $M$ of size at most $k-1$, which can be done in time $O(n^{k(k-1)})$.
To check if $H\in \mathcal{H}_{n,k}$, by Lemma \ref{lem22KKM} with $d=2$ and $L=L_{odd}$, we find the bipartitions for $L_{odd}$ in time $O(n^{k+1})$. Then for each bipartition $\calP=\{V_1, V_2\}$, we check if $n/k-|V_1|$ is odd in constant time.
Thus, the overall running time is $O(n^{2^{k-1}k+1})$.
\end{proof}

\subsection{A faster algorithm}

 An \emph{$s$-certificate} for $H$ is an insoluble full pair $(\calP, L)$ for which some set of $s$ vertices intersects every edge $e\in H$ with $\bfi_{\calP}(e)\notin L$. Note that if a full pair $(\calP, L)$ is soluble, then it is not an $s$-certificate for any $s$.
Recall that we allow the partition of a full pair to have $k$ parts and in contrast, the partition of a full pair in \cite{KKM13} has at most $k-1$ parts. Modifying the proof of \cite[Lemma 8.14]{KKM13} gives the following lemma.

\begin{lemma}\cite{KKM13}\label{lem84}
Suppose that $k\ge 3$ and $H$ is a $k$-graph such that there is no $2k(k-2)$-certificate for $H$. Then every full pair for $H$ is soluble.
\end{lemma}

Now we give the following structural theorem.

\begin{theorem} \label{thm:big}
Suppose $1/n_0 \ll \{\beta, \mu\} \ll \r \ll 1/k$.
Let $H$ be a $k$-graph on $n\ge n_0$ vertices such that $\delta_{k-1}(H)\ge n/k$ with $\calP_0$ and $\calP_0'$ found by Lemma \ref{lem:PL}. Then the following properties are equivalent.
\begin{enumerate}[(i)]
\item $H$ contains a perfect matching.
\item There is no $2k(k-2)$-certificate for $H$. 
\item The full pair $(\calP_0', L_{\calP_0'}^{\mu}(H))$ is soluble and $H\notin \mathcal H_{n,k}$.
\end{enumerate}
\end{theorem}

\begin{proof}
We will show that $(i)\Rightarrow (ii)\Rightarrow (iii)\Rightarrow (i)$. Note that the proof of $(i)\Rightarrow (ii)$ is the same as the forward implication of proof of Theorem \ref{thm:PM} and $(iii)\Rightarrow (i)$ by Theorem \ref{thm:PM}. It remains to show $(ii)\Rightarrow (iii)$. Assume that there is no $2k(k-2)$-certificate for $H$, then by Lemma \ref{lem84}, every full pair for $H$ is soluble.

Since $(\calP_0', L_{\calP_0'}^{\mu}(H))$ is a full pair, it is soluble. Second, assume to the contrary, that $H\in \mathcal H_{n,k}$. Then there is a partition $\calP_1=\{X, Y\}$ of $V(H)$ such that $L_{\calP_1}(H)\subseteq L_{odd}$ and $|X| - n/k$ is odd. Consider any $(k-1)$-set $S$ with $|S\cap X|=a$ for some even $0\le a\le k$, since $H$ contains no even edge and $\delta_{k-1}(H)>0$, we have $(a+1, k-a-1)\in I_{\calP_1}(H)$ and thus $L_{\calP_1}(H)= L_{odd}$. 
Also, $L_{\calP_1}(H)=L_{odd}$ is transferral-free and thus $(\calP_1, L_{\calP_1}(H))$ is a full pair. Note that by definition, the first coordinate of each $\bfi\in I_{\calP_1}(H)$ is odd and thus for any $(x,y)\in L_{\calP_1}(H)$, we have $k\mid (x+y)$ and $x \equiv (x+y)/k$ (mod 2).
So $\bfi_{\calP_1}(V)=(|X|, |Y|)\notin L_{\calP_1}(H)$ because $|X| - n/k$ is odd. Moreover, fix any edge $e$ of $H$ with $\bfi_{\calP_1}(e)=(a, k-a)$ for some odd $a\in [k]$, then 
$\bfi_{\calP_1}(V\setminus e)=(|X|-a, |Y| - k+a)\notin L_{\calP_1}(H)$ because $|X| - a - (n-k)/k = |X| - n/k - a +1$ is odd. So for any matching $M$ of size at most 1, $\bfi_{\calP_1}(V(H)\setminus V(M))\notin L_{\calP_1}(H)$. Thus, $(\calP_1, L_{\calP_1}(H))$ is an insoluble full pair, a contradiction.
\end{proof}

\begin{proof}[Proof of Theorem \ref{thm:main}]
Let $H$ be an $n$-vertex $k$-graph with $\delta_{k-1}(H)\ge n/k$. 
Note that it is trivial to determine the existence of a perfect matching if $n<n_0$ given by Theorem \ref{thm:big}.
If $n\ge n_0$, by Theorem \ref{thm:big}, to determine if $H$ contains a perfect matching, we only need to search the existence of a $2k(k-2)$-certificate for $H$. This can be done by Procedure DeterminePM constructed in \cite{KKM13}.
We estimate the running time as follows. There are at most $n^{2k(k-2)}$ choices of sets $S$, and these can be generated in time $O(n^{2k(k-2)})$. Also, there are only a constant number of choices for $d$ and $L$, and these can be generated in constant time.
For each choice of $S, d$ and $L$, we apply Procedure ListPartitions on $H[V\setminus S]$ and then add the vertices of $S$ arbitrarily to the partition we obtained. This generates the list of partitions $\calP$ in time $O(n^{k+1})$ by Lemma \ref{lem22KKM}. Furthermore, the number of choices for $\calP$ is constant, and for each one it takes time $O(n^{k(k-1)})$ to check the existence of the matching $M$ of size at most $d-1$ such that $\bfi_{\calP}(V(H)\setminus V(M))\in L_{\calP}(H)$. Note that $k(k-1)>k+1$ for all $k\ge 3$ and the total running time is $O(n^{2k(k-2)+k(k-1)}) = O(n^{3k^2 - 5k})$.
\end{proof}

\section{Concluding remarks}

Let $\DPM(k, m)$ be the decision problem of determining whether a $k$-graph $H$ with $\delta_{k-1}(H)\ge m$ contains a perfect matching. Our result implies that $\DPM(k,m)$ is in P for $m\ge n/k$ and the result in \cite{Szy13} shows that $\DPM(k, n/k - \r n)$ is NP-complete for any $\r>0$.
We remark that the argument in \cite{Szy13} actually shows that $\DPM(k, n/k - n^c)$ is NP-complete for any $c>0$.
Thus, $\DPM(k,m)$ is only unknown for $n/k - n^c\le m< n/k$.

In \cite{KKM13}, a polynomial-time algorithm for finding a perfect matching is also constructed. The problem of finding a perfect matching (in polynomial time) in the case when $\delta_{k-1}(H)\ge n/k$ remains open. 

As mentioned in \cite{KRS10}, it is also interesting to ask the corresponding decision problems for perfect matchings under other degree conditions, namely, $\delta_{d}(H)$ for $1\le d<k-1$, provided a gap between the thresholds for perfect matchings and perfect fractional matchings.

\section{Acknowledgement}
This work was done while the author was a PhD student under the supervision of Yi Zhao at Georgia State University.
The author would like to thank Hi\d{\^{e}}p H\`an and Yi Zhao for helpful discussions on the project.
The author also would like to thank an anonymous referee, Richard Mycroft and Yi Zhao for comments that improve the presentation of the paper.

\bibliographystyle{plain}
\bibliography{Jun2016}

\end{document}